\documentclass[twoside,11pt]{article}

\usepackage{graphicx, subfigure}
\usepackage[top=1in,bottom=1in,left=1in,right=1in]{geometry}
\usepackage[sort&compress, numbers]{natbib} 
\usepackage{amsfonts, amsmath, amssymb, amsthm, constants, bbm}
\usepackage{MnSymbol}
\usepackage{color}
\definecolor{darkred}{RGB}{100,0,0}
\definecolor{darkgreen}{RGB}{0,100,0}
\definecolor{darkblue}{RGB}{0,0,150}

\usepackage{hyperref}
\hypersetup{colorlinks=true, linkcolor=darkred, citecolor=darkgreen, urlcolor=darkblue}
\usepackage{url}

\newtheorem{thm}{Theorem}
\newtheorem{prp}{Proposition}
\newtheorem{lem}{Lemma}
\newtheorem{cor}{Corollary}

\theoremstyle{remark}

\newtheorem{rem}{Remark}

\def\beq{\begin{equation}} % \setcounter{equation}{1}}
\def\eeq{\end{equation}}
\def\beqn{\begin{eqnarray*}}
\def\eeqn{\end{eqnarray*}}
\def\Bitem{\begin{itemize}\setlength{\itemsep}{.2in}}
\def\bitem{\begin{itemize}\setlength{\itemsep}{.05in}}
\def\eitem{\end{itemize}}
\def\Benum{\begin{enumerate}\setlength{\itemsep}{.2in}}
\def\benum{\begin{enumerate}\setlength{\itemsep}{.05in}}
\def\eenum{\end{enumerate}}
\def\bmult{\begin{multline*}}
\def\emult{\end{multline*}}
\def\bcenter{\begin{center}}
\def\ecenter{\end{center}}
\def\bframe{\begin{frame}}
\def\eframe{\end{frame}}

\newcommand{\thmref}[1]{Theorem~\ref{thm:#1}}
\newcommand{\prpref}[1]{Proposition~\ref{prp:#1}}
\newcommand{\corref}[1]{Corollary~\ref{cor:#1}}
\newcommand{\lemref}[1]{Lemma~\ref{lem:#1}}
\newcommand{\secref}[1]{Section~\ref{sec:#1}}

\DeclareMathOperator{\diam}{diam}

\def\cG{\mathcal{G}}
\def\cH{\mathcal{H}}

\def\cK{\mathcal{K}}

\def\cP{\mathcal{P}}
\def\cQ{\mathcal{Q}}

\def\cS{\mathcal{S}}

\def\bbI{\mathbb{I}}

\def\bbR{\mathbb{R}}

\def\bbZ{\mathbb{Z}}

\newcommand{\E}{\operatorname{\mathbb{E}}}
\renewcommand{\P}{\operatorname{\mathbb{P}}}
\newcommand{\Var}{\operatorname{Var}}
\newcommand{\Cov}{\operatorname{Cov}}

\newcommand{\<}{\langle}
\renewcommand{\>}{\rangle}

\def\eps{\varepsilon}

\def\1{\mathbbm{1}}
\newcommand{\IND}[1]{\bbI\{ #1 \}}

\definecolor{purple}{rgb}{0.4,.1,.9}

\begin{document}
\thispagestyle{empty}

\title{Remember the Curse of Dimensionality: \\
The Case of Goodness-of-Fit Testing in Arbitrary Dimension}

\author{
Ery Arias-Castro\footnote{ Department of Mathematics, University of California, San Diego}
\and
Bruno Pelletier\footnote{ D\'epartement de Math\'ematiques, IRMAR -- UMR CNRS 6625, Universit\'e Rennes II}
\and
Venkatesh Saligrama\footnote{ Department of Electrical and Computer Engineering, Boston University}
}
\date{}

\maketitle

\begin{abstract}
Despite a substantial literature on nonparametric two-sample goodness-of-fit testing in arbitrary dimensions spanning decades, there is no mention there of any curse of dimensionality.  
Only more recently \citet{ramdas2015decreasing} have discussed this issue in the context of kernel methods by showing that their performance degrades with the dimension even when the underlying distributions are isotropic Gaussians.
We take a minimax perspective and follow in the footsteps of \citet{ingster1987minimax} to derive the minimax rate in arbitrary dimension when the discrepancy is measured in the $L^2$ metric.  That rate is revealed to be nonparametric and exhibit a prototypical curse of dimensionality.  
We further extend Ingster's work to show that the chi-squared test achieves the minimax rate.  Moreover, we show that the test can be made to work when the distributions have support of low intrinsic dimension.  Finally, inspired by  \citet{ingster2000adaptive}, we consider a multiscale version of the chi-square test which can adapt to unknown smoothness and/or unknown intrinsic dimensionality without much loss in power.  
\end{abstract}

\section{Introduction} \label{sec:intro}
We consider the multivariate two-sample goodness-of-fit testing problem in a nonparametric setting. In the two-sample goodness-of-fit problem we are given two datasets, $X_1,X_2,\ldots,X_m$ IID with unknown distribution $F$ and $Y_1, \dots, Y_n$ IID with unknown distribution $G$, and the goal is to determine whether or not $F=G$. 
In the classical statistics literature, this question has been studied extensively in univariate setting, with well-known tests such as the Kolmogorov-Smirnov test \cite{kolmogorov1933sulla, smirnov1939estimation}, the number-of-runs test \cite{MR0002083}, or the longest-run test \cite{MR0004453}.
Work on multivariate two-sample tests can be traced to~\cite{hotelling1951}. 

Comparatively, much less work has been devoted to the multivariate setting.
High-dimensional two-sample goodness-of-fit problems arise naturally in medical, social and financial applications. For instance, in medical applications, the behavior or response to a drug can manifest in terms of a diverse set of symptoms, and the goal is to detect differences among these multitude of symptoms. Cross-cultural differences in social sciences are often measured with respect to a number of different behavioral factors. Finally, inferring differences/changes in stock market trends is generally based not only on macro/microeconomic but also a number of prevailing political and social factors.
There is, therefore, a need for understanding the fundamental achievable limits of multivariate two-sample goodness-of-fit testing.  Here we are mostly interested in the nonparametric situation.

Despite a growing literature on nonparametric two-sample goodness-of-fit testing in arbitrary dimensions, there is little mention of a curse of dimensionality. This confusion is compounded by the fact that parametric rates are derived in some of these works.
 \citet{ramdas2015decreasing, ramdas2015adaptivity} have challenged these claims by showing that the performance of some recent popular kernel methods (introduced below) degrades as the dimension increases, even when the underlying distributions are as simple as isotropic Gaussians.
It turns out, as we elaborate upon below, that these parametric rates are obtained for directional alternatives. While analyzing directional alternatives could be meaningful in some cases such as the one-sample problem and other special cases where prior information is available, it does not appear to be meaningful in the context of a nonparametric two-sample problem. 
In contrast, taking a classical minimax perspective reveals that a parametric rate is not achievable, as already established by \citet{ingster1987minimax} in the one-dimensional setting.

In this paper, we extend \citeauthor{ingster1987minimax}'s results to arbitrary dimensions.  We assume, as he did, that the underlying distributions have H\"older smooth densities supported on the unit hypercube, and derive the minimax rate with respect to the $L^2$ metric.  This rate is not only nonparametric, but exhibits a prototypical curse of dimensionality. In the process, we show that the chi-squared test with a proper choice of bin size achieves the minimax rate. 
Furthermore, we show that, with proper tuning, the chi-squared test is also able to achieve the minimax rate over densities with lower-dimensional support.  
Finally, inspired by \citet{ingster2000adaptive}, we consider a multiscale version of the chi-squared test which is essentially parameter-free and can adapt to unknown smoothness and/or unknown intrinsic dimensionality without much loss in power.

%Two-sample goodness-of-fit testing is an important line of research in statistics, with well-known tests such as the Kolmogorov-Smirnov test \cite{kolmogorov1933sulla,smirnov1939estimation}, based on comparing the empirical distribution functions; the number-of-runs test \cite{MR0002083}; or the longest-run test \cite{MR0004453}. 

\subsection{The literature}
We can broadly classify the existing literature on the nonparametric two-sample goodness-of-fit tests into two categories: (a) Tests based on metrics; (b) Tests based on graph constructions. 

\subsubsection{Tests based on metrics} \label{sec:intro-metrics}
Recall that we have two independent samples, $X_1, \dots, X_m$ IID with distribution $F$ and $Y_1, \dots, Y_n$ IID with distribution $G$, where $F$ and $G$ are Borel measures on $\bbR^d$.  
\citet{bickel1969distribution} proposes a direct extension of the Kolmogorov-Smirnov test.  His proposal is a special case of tests of the form $\gamma_\cH(F_m, G_n)$, where $F_m$ and $G_n$ denote the empirical distributions of the $X$ and $Y$ samples, respectively, and
\beq\label{gamma}
\gamma_\cH(P, Q) := \sup_{h \in \cH} \textstyle \big| \int h\,{\rm d}P - \int h\,{\rm d}Q \big|,
\eeq
for an appropriate class of test functions $\cH$.  By varying $\cH$, besides the Kolmogorov distance, one can get the total variation distance and the Wasserstein distance, among others. \citet{sriperumbudur2010hilbert} provide a comprehensive overview.  
The metric \eqref{gamma} may be difficult to compute in general, even for discrete measures, because of the supremum over $\cH$.  However, by taking $\cH$ to be a reproducible kernel Hilbert space (RKHS), as advocated in \cite{berlinet2011reproducing, gretton2006kernel, smola2007hilbert}, then
\beq\label{gamma-rkhs}
\begin{split}
\gamma_\cH(P, Q) 
&= \int_{}\int \psi(x,y) P({\rm d}x) P({\rm d}y) + \int_{} \int \psi(x,y) Q({\rm d}x) Q({\rm d}y) \\
&\quad - 2 \int_{} \int \psi(x,y) P({\rm d}x) Q({\rm d}y),
\end{split}\eeq
where $\psi$ denotes the kernel defining $\cH$.
The sample version is the following computationally-friendly $U$-statistic
\beq\begin{split}
\gamma_\cH(F_m, G_n) 
&= \frac2{m(m-1)} \sum_{i=1}^{m-1} \sum_{j = i+1}^m \psi(X_i, X_j) + \frac2{n(n-1)} \sum_{i=1}^{n-1} \sum_{j = i+1}^n \psi(Y_i, Y_j) \\
&\quad - \frac2{mn} \sum_{i=1}^{m} \sum_{j = 1}^n \psi(X_i, Y_j).
\end{split}\eeq
When $\psi$ is bounded, this is a $(m \wedge n)^{1/2}$-consistent estimator for $\gamma_\cH(F, G)$.  The question then becomes whether $\gamma_\cH(F, G)$ is a true metric, a question addressed in \cite{sriperumbudur2010hilbert}.  

This line of work is intimately related to that of \citet{zinger1992characterization}.  The notion of N-distance that developed out of this work \cite{klebanov2005n} is exactly of the form \eqref{gamma-rkhs}, therefore coinciding with the pseudo-metric $\gamma_\cH$ when $\cH$ is an RKHS.
Applications to the two-sample problem are developed in \cite{szekely2004testing,bakshaev2009goodness}.

\subsubsection{Tests based on graph constructions} \label{sec:intro-graphs}
In a disjoint line of work, another class of tests has developed based on various graph constructions.  This goes back, at least, to the work of \citet{friedman1973nonparametric}.  There, for each point in the combined sample the number of $X$'s among its $K$-nearest neighbors is recorded.  This results in two distributions on $\{0,1,\dots,K\}$, $\hat\pi_X$ and $\hat\pi_Y$, where $\hat \pi_X(k)$ (resp.~$\hat \pi_Y(k)$) denotes the number of $X$'s (resp.~$Y$'s) having exactly $k$ other $X$'s among their $K$-nearest neighbors.  The distributions $\hat \pi_X$ and $\hat \pi_Y$ are then compared in some way, resulting in a test.  See also \cite{rogers1976some, hall2002permutation}.
Although it does not cover all the possibilities, many of the subsequent proposals in this line of work can be framed as follows.  Let $t = m+n$ denote the total sample size.  Let $\cG$ be a directed graph with node set $\{1, \dots, t\}$ indexing the combined sample, $\{Z_1, \dots, Z_{t}\}$ with $Z_k = X_k$ if $k \le m$ and $Z_k = Y_{k - m}$ if $k \ge m+1$.  We write $i \to j$ when node $i$ points to node $j$ in $\cG$.  Consider rejecting for {\em small} values of
\beq\label{crossmatch}
\chi_\cG = \# \{i \to j : i \le m, j \ge m+1\} + \# \{j \to i : i \le m, j \ge m+1\},
\eeq
which is the number of neighbors in the graph from different samples.
If the graph $\cG$ is the $K$-nearest neighbor graph --- where $i \to j$ if $Z_j$ is among the $K$-nearest neighbors of $Z_i$ in Euclidean distance --- and we assume that all the $Z$'s are distinct, then the resulting test is that of \citet{schilling1986multivariate}, a special case of the general approach of \citet{friedman1973nonparametric}. 
If the graph $\cG$ is a minimum spanning tree (starting with the complete graph weighted by the Euclidean distances), then the resulting test is the multivariate number-of-runs test of \citet{friedman1979multivariate}.  
If the graph $\cG$ is a minimum distance non-bipartite matching, then the resulting test is that of \citet{rosenbaum2005exact}.
We refer the reader to \cite{bhattacharya2015general} for additional references and recent theoretical developments

\subsection{The curse of dimensionality} \label{sec:intro-curse}
Although the setting is nonparametric, surprisingly, there is no discussion of a curse of dimensionality in this literature.  
We argue here --- and develop this further in the rest of the paper --- that there is a bonafide curse of dimensionality.  Indeed, suppose that $F$ and $G$ are supported on the unit hypercube $[0,1]^d$ and assume furthermore that they have densities $f$ and $g$ with respect to the Lebesgue measure that can be taken such that $f \le C$ and $g \le C$ for some constant $C < \infty$.  
Fix $\eps \in (0,1/2)$.  Then the chances of {\em not} observing any $X$ in $[\eps, 1-\eps]^d$ out of a sample of size $m$ are
\begin{align}
\P(\forall i \in [m]: X_i \notin [\eps, 1-\eps]^d)
&= \P(X_1 \notin [\eps, 1-\eps]^d)^m \\
&= \left(\int_{[0,1]^d\setminus [\eps, 1-\eps]^d} f(x) {\rm d}x\right)^m \\
&\ge \left(1 - C (1-2\eps)^d\right)^m \\
&\to 1, \text{ when } m (1-2\eps)^d \to 0.
\end{align}
The condition on $m$ and $d$ holds, for example, when $d \gg \log m$.  (Of course, the same derivations apply for the $Y$ sample as well.)  We conclude that, when the dimension is a little more than logarithmic in the sample size, the inner hypercube $[\eps, 1-\eps]^d$ is very likely empty of data points.  This is prototypical of a curse of dimensionality and the consequences are completely standard: if there are no data points in $[\eps, 1-\eps]^d$ (for example, $[1/4, 3/4]^d$ when $\eps = 1/4$) then, obviously, we cannot distinguish $f$ and $g$ if they agree outside of that hypercube.  And typical smoothness assumptions on $f$ and $g$ (made explicit later) allow for this to happen.

\subsection{Directional alternatives and minimaxity} \label{sec:param_rate_lit}
Recall that $t = m+n$ denotes the total sample size and consider an asymptotic setting where $m \asymp n \asymp t \to \infty$.  In this context, many of the various tests proposed in the literature just cited are shown to be consistent against all alternatives, and sometimes claimed to be $t^{1/2}$-consistent, as in  \cite{hall2002permutation,schilling1986multivariate} and also implicitly in \cite{sriperumbudur2010hilbert}, among others. Furthermore, \cite{hall2002permutation} places other conventional distributional tests including Mann-Whitney rank test, Kolmogorov-Smirnov and Cramer-von Mises tests in the same league as their own permutation tests. Note that $t^{1/2}$ is the parametric rate in this context.
We have already mentioned the work of \citet{ramdas2015adaptivity,ramdas2015decreasing}, who challenge these claims by observing and establishing the fact that the performance of the two-sample test based on the Gaussian kernel (among others) can degrade with the dimension even in (parametric) location models.
In the process, \citeauthor{ramdas2015decreasing} question the choice of alternatives made in various papers advocating for kernel-based methods.

As \citet{hall2002permutation} argue, these parametric rates must be understood in the directional sense.  As before, let $f$ and $g$ denote the densities of $F$ and $G$ with respect to the Lebesgue measure.  A directional alternative is of the form $g = f + \eps_t h$ (where necessarily $\int h\, {\rm d}\nu= 0$) and a test of the form $\{T \ge \tau_t\}$ is $t^{1/2}$-consistent in the direction of $h$ if $\P_0(T \ge \tau_t) \to 0$ and $\P_{\eps_t}(T \ge \tau_t) \to 1$ when $\eps_t \gg 1/\sqrt{t}$, where $\P_\eps$ denotes here the law when the samples come from $f$ and $g = f + \eps h$, with $h$ given.
The tests in \cite{hall2002permutation,schilling1986multivariate} are shown to be $t^{1/2}$-consistent in all directions under additional (but mild) regularity assumptions.  

We find it hard to motivate directional consistency, particularly for nonparametric two-sample goodness-of-fit problems where neither of the distributions $F,\,G$ are known. Unlike one-sample tests where local alternatives can be described with respect to the known distribution, for two-sample tests it is somewhat unclear as to how to characterize such local alternatives. Furthermore, a large-sample analysis is meant to elucidate what happens in practice when the samples are sufficiently large.  With that in mind, what does it mean for the direction $h$ to remain fixed as $m,n \to \infty$?  In addition this notion of performance can also be misleading.  First, the rate appears parametric in a typical nonparametric setting.  Second, in the present setting, it hides the fact that there is a bonafide (in fact, prototypical) curse of dimensionality, as argued earlier.  We turn to the notion of minimax performance, which is much more commonly used to quantify the hardness of a given statistical problem.  
(\citet{baraud2003adaptive} provide a deeper discussion of directional alternatives from a minimax perspective.)

\subsection{Our contribution}
Our main purpose here is to clarify the situation by contributing the following:
\bitem
\item {\em Minimax lower bound.}
%This is true even in dimension $d=1$ where, for example, the minimax rate is known to be $O(m \wedge n)^{-2/5}$ when $F$ and $G$ have Lipschitz densities with respect to the Lebesgue measure.  
We derive a minimax lower bound under H\"older regularity in arbitrary dimension following the work of \citet{ingster1987minimax}, who considers the one-sample setting in dimension $d=1$.
We do so for the one-sample and two-sample settings.  
In each case, the minimax rate exhibits a typical curse of dimensionality.  
%We briefly discuss the case where the regularity is unknown based on the work of \citet{ingster2000adaptive}.
\item {\em Chi-squared test.}
The minimax performance of the various tests mentioned earlier is, as far as we know, unknown.  
Extending the work of \citet{ingster1987minimax}, we show that a properly tuned chi-squared test achieves the minimax rate.  We do so for the one-sample and two-sample settings.    
\item {\em Low-intrinsic dimension.}
As is now standard in high-dimensional settings, we consider the case where the supports of $F$ and $G$ have low intrinsic dimension.  We argue that the most relevant setting is when $F$ and $G$ are supported on the same surface and, in this context, show that a properly tuned chi-squared test achieves the minimax rate.
\item {\em Adaptation to unknown smoothness and/or intrinsic dimensionality.}
As in the work of \citet{ingster2000adaptive}, we propose a multiscale version of the chi-squared test which is essentially parameter-free and is able to adapt to an unknown degree of  smoothness and/or intrinsic dimensionality without much loss in power.
\eitem  

\subsection{Notation}

For two vectors $a = (a_1,\dots,a_d)$ and $b = (b_1,\dots,b_d)$ in $\bbR^d$, with $a_j \le b_j$ for all $j$, define the following hyperrectangle
\beq\label{hyperrectangle}
[a,b] := \prod_{j=1}^d [a_j,b_j].
\eeq

\section{The one-sample goodness-of-fit problem}
\label{sec:one}

We start by extending the work of \citet{ingster1987minimax}, who considers the one-sample setting in dimension $d=1$, to an arbitrary dimension $d \ge 1$.

In the one-sample setting, we have at our disposal one sample $X_1, \dots, X_m$ IID with distribution $F$, with density $f$ with respect to the Lebesgue measure $\nu$ on the unit hypercube, $[0,1]^d$.  
The goal is to test
\beq
H_0 : f = f_0 \quad \text{versus} \quad H_1 : \delta(f, f_0) \ge \eps,
\eeq
for some pseudo-metric $\delta$.  (Of course, $f = f_0$ is understood modulo a set of $\nu$-measure zero.) 
We will work throughout with square-integrable densities, for which the $L^2$-metric is appropriate,
\beq\label{L2}
\delta(f,g) = \|f -g\|_2 := \sqrt{\int (f-g)^2 {\rm d}\nu}.
\eeq
Following \citeauthor{ingster1987minimax}, we focus on the case where $f_0$ is the uniform distribution on $[0,1]^d$, knowing that the arguments extend without pain to almost any other choice of null distribution supported on the hypercube.

\subsection{Risk and minimax lower bound}
\label{sec:intrinsicer1}

A test, $\phi$, is a Borel measurable function of the data --- meaning $\bbR^d \times \cdots \bbR^d$, $m$ times, when $m$ is the sample size --- with values in $[0,1]$, which gives the probability of rejecting the null hypothesis.
Let $\cH$ denote a class of real-valued functions on $\bbR^d$ and let $\delta$ be a pseudo-metric on $\cH$.  For $\eps > 0$, define the worst-case risk of a test $\phi$ as
\beq\label{risk1}
R_\eps^{(m)}(\phi; f_0; \cH) = \E^{(m)}_{f_0}[\phi] + \sup\Big\{\E^{(m)}_{f}[1-\phi] : f \in \cH; \delta(f,f_0) \ge \eps\Big\},
\eeq
The first (resp.~second) supremum is the largest probability of type I (resp.~II) error.  The minimax risk is
\beq\label{minimax-risk1}
R_\eps^{(m)}(f_0; \cH) = \inf_\phi R_\eps^{(m)}(\phi; f_0; \cH).
\eeq

In nonparametric settings such as the present one, it is customary to make regularity assumptions on the underlying distributions.  A typical assumption is that of smoothness \cite{ingster1987minimax,ingster1993asymptotically}.  We follow \citeauthor{ingster1987minimax} and work with H\"older regularity.  
For $s > 0$, let $\lfloor s \rfloor$ denote the largest integer strictly smaller than $s$.  Let $\cH_s^d(L)$ denote the class of functions $f : [0,1]^d \mapsto \bbR$ such that $f$ has a derivative of order $\lfloor s \rfloor$ which satisfies
\beq\label{H}
\big|f^{\lfloor s \rfloor}(x) - f^{\lfloor s \rfloor}(y)\big| \le L \|x - y\|^{s - \lfloor s \rfloor}, \quad \forall x,y \in [0,1]^d. 
\eeq
For convenience,\footnote{ It can be shown that a density $f$ satisfying \eqref{H} satisfies \eqref{H-deriv} but with $L$ replaced by a larger constant depending only on $(s,d,L)$.  This is not true of just any function satisfying \eqref{H} since that condition alone does not imply a uniform bound on the supnorm.} 
we add the assumption that 
\beq\label{H-deriv}
\|f^{(s')}\|_\infty \le L, \quad \forall s' \in \{0, \dots, \lfloor s \rfloor\}.
\eeq 
For example, the functions in $\cH^1(L)$ are Lipschitz with constant $L$.
A straightforward extension of the lower bound derived by \citeauthor{ingster1987minimax} leads to the following lower bound on the minimax rate for the one-sample problem in dimension $d$.  
Recall that we work with the uniform distribution, where $f_0 \equiv 1$ on $[0,1]^d$, and leave $f_0$ implicit in \eqref{risk1} and \eqref{minimax-risk1}.

\begin{thm} \label{thm:lower1}
For the one-sample problem under known H\"older regularity, there is a constant $c > 0$ depending only on $(s,d,L)$ such that
\beq\label{lower1-1}
R_\eps^{(m)}(\cH_s^d(L)) \ge 1/2, \quad \text{if } \eps \le c\, m^{-2s/(4s+d)}.
\eeq
\end{thm}
The proof is only provided for pedagogical reasons. 
This result quantifies the curse of dimensionality presented in \secref{intro-curse}.  In particular, we can see again that if $d \gg \log m$, the upper bound on $\eps$ does not even approach~0 even as $m \to \infty$.

\subsection{The chi-squared test}
\label{sec:tests1}

In the one-sample setting in dimension $d=1$, when the H\"older regularity $s$ is known, \citet{ingster1987minimax} establishes that the chi-squared test with bin size\footnote{ We only consider partitions into bins of equal size.} $1/\kappa \asymp m^{2/(4s+1)}$ achieves the minimax rate.  

In dimension $d$, the chi-squared test rejects for large values of 
\beq\label{Gamma1}
\Gamma_\kappa^{\rm one} := \sum_{k \in [\kappa]^d} (M_{k,\kappa} - m \kappa^{-d})^2,
\eeq
where, for an integer $\kappa \ge 1$ and $k \in [\kappa]^d$, we define the bin counts
\beq
M_{k,\kappa} = \#\big\{i \in [m] : X_i \in ((k-1)/\kappa, k/\kappa]\big\}.
\eeq
When $s$ is known, we set the bin size to be $1/\kappa$, where
\beq\label{kappa}
\kappa = \kappa(s,d) := \lfloor m^{2/(4s+d)} \rfloor.
\eeq
Note that this choice of bin size depends on the degree of smoothness.

\begin{thm} \label{thm:chi1}
In the one-sample problem under known H\"older regularity, consider the chi-squared test $\phi_{\kappa, \tau} := \IND{\Gamma_\kappa^{\rm one} > \tau}$.  
There are constants $c_1$ depending only on $(d,L)$ and $c_2$ depending only on $(s,d,L)$ such that, by choosing $\kappa$ as in \eqref{kappa} and $\tau = m + a m \kappa^{-d/2}$ with $a \ge 1$,
\beq
\bar R_\eps^{(m)}(\phi_{\kappa, \tau}; \cH_s^d(L)) \le c_1/a^2, \quad \text{if }\eps \ge c_2 a m^{-2s/(4s+d)}.
\eeq
\end{thm}

In particular, the performance of the chi-squared test matches the minimax lower bound of \thmref{lower1} up to a multiplicative constant.  

\begin{rem}
We do not provide a proof of \thmref{chi1} since it is analogous (and in fact simpler) than the proof of \thmref{chi2}, the corresponding result in the two-sample setting.
\end{rem}

\begin{rem}[Calibration by Monte Carlo]
Although we provide a critical value for the test statistic \eqref{Gamma1}, in practice it can be very approximate and a calibration by Monte Carlo simulation (under the null distribution) is recommended.  As usual, the calibration is meant to control the probability of type I error at a given level, instead of controlling the risk (which in principle requires some information on the alternative).
\end{rem}

\section{The two-sample goodness-of-fit problem}
\label{sec:two}

We now turn to the two-sample setting.  
Here we have at our disposal two independent samples, $X_1, \dots, X_m$ IID with distribution $F$ and $Y_1, \dots, Y_n$ IID with distribution $G$.  
We assume that $F$ and $G$ have densities $f$ and $g$ with respect to the Lebesgue measure on $[0,1]^d$.
The goal is to test
\beq\label{problem}
H_0 : f = g \quad \text{versus} \quad H_1 : \delta(f,g) \ge \eps.
\eeq
for some pseudo-metric $\delta$.  
As before, we use the $L^2$-metric \eqref{L2} and assume that $f$ and $g$ are in the H\"older class parameterized by $(s,d,L)$.  

\subsection{Risk and minimax lower bound}
\label{sec:intrinsicer2}

A test, $\phi$, is now a Borel measurable function of the data  --- which now consists of $m$ points from the $X$-sample and $n$ points from the $Y$-sample --- with values in $[0,1]$.
Let $\cH$ denote a class of real-valued functions on $\bbR^d$ and let $\delta$ be a pseudo-metric on $\cH$.  For $\eps > 0$, define the worst-case risk of a test $\phi$ as
\beq\label{risk}
R_\eps^{(m,n)}(\phi; \cH) = \sup\Big\{\E^{(m,n)}_{f,f}[\phi] : f \in \cH\Big\} + \sup\Big\{\E^{(m,n)}_{f,g}[1-\phi] : f,g \in \cH; \delta(f,g) \ge \eps\Big\}.
\eeq
The minimax risk is
\beq\label{minimax-risk}
R_\eps^{(m,n)}(\cH) = \inf_\phi R_\eps^{(m,n)}(\phi; \cH),
\eeq
where the infimum is over all tests $\phi$.

We obtain a minimax lower bound by reducing the two-sample problem to the one-sample problem.  Intuitively, it is clear that the former is at least as hard as the latter, which in essence corresponds to the case where one of the samples (say, the $Y$-sample) is infinite, so that the corresponding density ($g$ for the $Y$-sample) is known in principle.

\begin{lem} \label{lem:reduction}
For any class $\cH$, any pseudo-metric $\delta$, any $\eps > 0$, any density function $f_0 \in \cH$, and any integers $m, n \ge 1$,
\beq
R_\eps^{(m)}(f_0; \cH) \le R_\eps^{(m,n)}(\cH).
\eeq
\end{lem}

\lemref{reduction} and \thmref{lower1}, together, lead to the following.

\begin{thm} \label{thm:lower2}
For the two-sample problem under known H\"older regularity, there is a constant $c > 0$ depending only on $(s,d,L)$ such that
\beq\label{lower2-1}
R_\eps^{(m,n)}(\cH_s^d(L)) \ge 1/2, \quad \text{if } \eps \le c (m \wedge n)^{-2s/(4s+d)}.
\eeq
\end{thm}

\subsection{The chi-squared test}
\label{sec:tests2}

For an integer $\kappa \ge 1$ and $k \in [\kappa]^d$, define the bin counts
\beq\label{bin-counts}
\begin{split}
M_{k,\kappa} &= \#\big\{i \in [m] : X_i \in ((k-1)/\kappa, k/\kappa]\big\}, \\
N_{k,\kappa} &= \#\big\{j \in [n] : Y_j \in ((k-1)/\kappa, k/\kappa]\big\}.
\end{split}
\eeq
%In the two-sample setting, the chi-squared test is typically defined as the test rejecting for large values of
%\beq\label{chi0}
%\sum_{k \in [\kappa]^d} \frac{(n M_{k,\kappa} - m N_{k,\kappa})^2}{M_{k,\kappa} + N_{k,\kappa}}.
%\eeq
For simplicity, we work instead with the {\em unnormalized} chi-squared test, which rejects for large values of  
\beq
\sum_{k \in [\kappa]^d} (n M_{k,\kappa} - m N_{k,\kappa})^2.
\eeq
In fact, to further simplify the setting, we focus on the case where $m = n$, in which case the test rejects for large values of
\beq\label{Gamma2}
\Gamma_\kappa := \sum_{k \in [\kappa]^d} (M_{k,\kappa} - N_{k,\kappa})^2,
\eeq
In principle, if one simply wants a test that matches the minimax lower bound stated in \thmref{lower2}, in what follows one can simply assume that $m \le n$ (without loss of generality) and that $n-m$ observations from the $Y$ sample have been discarded.  Although no one would do this in practice, the resulting chi-squared test achieves the minimax rate.

\begin{thm} \label{thm:chi2}
In the two-sample problem under known H\"older regularity, consider the chi-squared test $\phi_{\kappa, \tau} := \IND{\Gamma_\kappa > \tau}$.  
There are constants $c_1$ depending only on $(d,L)$ and $c_2$ depending only on $(s,d,L)$ such that, by choosing $\kappa$ as in \eqref{kappa} and $\tau = 2m + a m \kappa^{-d/2}$ with $a \ge 1$,
\beq\label{chi2-1}
\bar R_\eps^{(m,m)}(\phi_{\kappa, \tau}; \cH_s^d(L)) \le c_1/a^2, \quad \text{if }\eps \ge c_2 a m^{-2s/(4s+d)}.
\eeq
\end{thm}

\begin{rem}[Calibration by permutation]
Although we provide a critical value for the test statistic \eqref{Gamma2}, in practice such values are not reliable and a calibration by permutation is recommended.  The calibration is again meant to control the probability of type I error at a given level.  
\end{rem}

\section{The assumption of low intrinsic dimension}
\label{sec:intrinsic}

When the data are high dimensional, a common approach to circumvent the curse of dimensionality --- which we now know is at play here --- is to assume the data have a low intrinsic dimensionality.  

In our context, in the one-sample setting, this translates into assuming that the null distribution has support of low dimension.  For example, it could be the uniform distribution on a surface.  In that case, the chi-square test could be based on bins that partition the surface, and not much is different.  
Thus we focus the discussion on the more complex setting of two samples.

We consider that setting under the assumption that distributions $F$ and $G$ have supports of low dimension.
This leads to two emblematic situations:  
\bitem
\item {\em Possibly distinct supports.}
Suppose $F$ (resp.~$G$) is the uniform distribution on a compact set $S$ (resp.~$T$) of $\bbR^d$.  In this case, the goal is to test $S = T$ versus $\delta(S, T) \ge \eps$ for some given pseudo-metric on a given class of sets.  For example, the class could be that of submanifolds of dimension $d_0$ (for some $d_0 \le d$,  perhaps unknown), without boundary and reach $\ge r$ \cite{MR0110078} and $\delta$ could be the minimum separation between $S$ and $T$.  
\item {\em Same support.}
Suppose $F$ and $G$ have densities $f$ and $g$ with respect to $\nu$, the uniform measure on a compact set $S$ in $\bbR^d$ of dimension $d_0 \le d$ (for some $d_0 \le d$,  perhaps unknown).  Here the goal is the same as in \eqref{problem}. 
\eitem

The first setting is closely related to some literature on manifold estimation \cite{genovese2012manifold,genovese2012minimax,kim2015tight}, detection \cite{MR2565828}, and clustering \cite{arias2011spectral}.  Based on that literature we speculate that to achieve a modicum of (minimax) optimality requires more specialized tests that take into account the geometry of the problem.  For the sake of cohesion, we will not address this situation here.

In the second setting, if $S$ is known the problem is very close to that of testing in $\bbR^{d_0}$, the only difference being that the binning would be custom built for $S$.  We now consider this setting when $S$ is unknown.  It happens that the chi-squared test, with proper calibration, is able to adapt to an unknown support.  (This is not too surprising since the same phenomenon arises in the context of regression, as established by \citet{kpotufe2011k}.)

Let $\cS_1^{d_0}(L_0)$ denote the class of sets in $[0,1]^d$ which, for each $\kappa > 0$, can be covered with $L_0 \kappa^{d_0}$ bins in a regular partition of $[0,1]^d$ into bins of size $1/\kappa$.  (This is the same partitions defining the chi-squared tests.)  By definition, the sets in $\cS_1^{d_0}(L_0)$ have box dimension at most $d_0$.  For $s \in (0,1]$, let $\cH_s^{d, d_0}(L, L_0)$ denote the sets of real-valued functions $f$ defined on some $S \in \cS_1^{d_0}(L_0)$ and that satisfy \eqref{H}-\eqref{H-deriv}.
Since $\cS_1^{d_0}(L_0)$ includes $S = [0, (L_0 \wedge 1)^{1/d_0}]^{d_0} \times \{0\}^d$, \thmref{lower2} applies with $d$ replaced by $d_0$.  \thmref{chi2} also applies, with $d$ replaced by $d_0$, since the arguments (detailed in \secref{proofs}) are entirely based on the bin size and the approximation within each bin.  
%Note that the bin size is set to
%\beq\label{kappa0}
%\kappa = \kappa(s,d_0) := \lfloor m^{2/(4s+d_0)} \rfloor,
%\eeq
%which requires knowledge of the intrinsic dimension $d_0$.  See the discussion in \secref{discussion}.

To define higher orders of smoothness requires the support set to be smooth enough.  Therefore, for an integer $q \ge 2$, let
\beq\label{Sq}
\cS_q^{d_0}(L_0) = \Big\{S = \sigma([0,1]^{d_0}) : \sigma = (\sigma_1, \dots, \sigma_d) \text{ with } \sigma_k \in \cH_q^{d_0}(L_0) \text{ and } {\rm range}(\sigma_k) \subset [0,1]\Big\}.  
\eeq
For $S \in \cS_q^{d_0}(L_0)$, we say that a function $f : S \mapsto \bbR$ is $\ell$ times differentiable, with $0 \le \ell \le q-1$, if there is $\sigma$ as in \eqref{Sq} such that $f \circ \sigma$ is $\ell$ times differentiable.  For $s \in (q-1,q]$, let $\cH_s^{d, d_0}(L, L_0)$ denote the sets of real-valued functions $f$ defined on some $S \in \cS_q(L_0)$ that satisfy \eqref{H}-\eqref{H-deriv}.
With this definition, once again, \thmref{lower2} and \thmref{chi2} apply with $d$ replaced by $d_0$.
(We note that all this remains true under the stronger requirement that $f \in \cH_s^{d}(L)$, which in particular requires that $f$ is defined on the whole $[0,1]^d$.  In that case, it suffices that $S \in \cS_1^{d_0}(L_0)$.)

We note that the choice of bin size that makes it possible to apply \thmref{chi2} depends not only on the degree of smoothness, as before, but also on the intrinsic dimension.

%\begin{rem}[Choice of $\kappa$]
%We have assumed that the bin size, $1/\kappa$, is chosen appropriately according to the intrinsic dimension $d_0$ of $S$ and the smoothness $s$ of the densities.  With this information, one would choose $\kappa = \kappa(s, d_0)$ as defined in \eqref{kappa}, which is the choice of $\kappa$ when the dimension is $d_0$.  If the intrinsic dimension of $S$ is unknown, one can resort to a multiscale test as described in \secref{unknown}.  
%\end{rem}

\section{Adaptation to unknown smoothness and/or intrinsic dimension} \label{sec:unknown}

We assumed in \secref{one} and \secref{two} that the degree of H\"older smoothness $s$ was known, and in \secref{intrinsic} we assumed that the dimension of the underlying support was known, and in both cases, this information was crucially used in the choice of bin size.  It is natural to ask what can be done when this information is not available.  

\citet{ingster2000adaptive} considers the situation where the smoothness is unknown, again focusing entirely on the one-sample setting in dimension $d=1$.  He proposes a multiscale chi-squared test, which consists in performing the chi-squared test of \secref{tests1} for each dyadic bin size in a certain range and applying a Bonferroni correction for multiple testing.  
He shows that the resulting test achieves the minimax rate, which he also derives for the setting where the smoothness is unknown and happens to differ from \eqref{lower1-1} by the fact that $m$ there is replaced by $m/\sqrt{\log\log m}$ factor in the upper bound on $\eps$.  In what follows, we also study a multiscale chi-squared test and derive a performance bound that, although crude compared to Ingster's, is sufficient to show that not much power is lost when the smoothness is unknown.  We will leave it as an exercise to the reader to check that the same test is also able to adapt to the intrinsic dimension of the underlying support in the context of \secref{intrinsic}.  We focus on the two-sample setting.

\subsection{Risk and minimax lower bound}
In the adaptive two-sample setting, for $\mathfrak{S} \subset (0,\infty)$ and $\eps = (\eps_s : s \in \mathfrak{S})$, define the risk of a test $\phi$ as 
\beq
\sup_{s \in \mathfrak{S}} R_{\eps_s}^{(m,n)}(\phi; \cH_s^d(L)),
\eeq 
and then the minimax risk as 
\beq
\inf_\phi \sup_{s \in \mathfrak{S}} R_{\eps_s}^{(m,n)}(\phi; \cH_s^d(L)),
\eeq
where this time the infimum is over all tests $\phi$ with only knowledge of $\mathfrak{S}$ and not the specific smoothness $s$ in that set.  (The parameters $\eps,L$ remain known, as before.)

Here we content ourselves with the obvious lower bound
\beq
\inf_\phi \sup_{s' \in \mathfrak{S}} R_{\eps_{s'}}^{(m,n)}(\phi; \cH_{s'}^d(L)) \ge 
\inf_\phi R_{\eps_s}^{(m,n)}(\phi; \cH_s^d(L)), \quad \forall s \in \mathfrak{S},
\eeq
so that, if $c(s)$ denotes the constant appearing in \eqref{lower2-1}, we have 
\beq\label{lower-adapt}
\inf_\phi \sup_{s \in \mathfrak{S}} R_{\eps_s}^{(m,n)}(\phi; \cH_s^d(L)) \ge 1/2, \quad \text{if there is $s \in \mathfrak{S}$ such that } \eps_s \le c(s) (m \wedge n)^{-2s/(4s+d)}.
\eeq 

\citet{ingster2000adaptive} derives a sharper bound for the case where $\mathfrak{S}$ is a compact interval of $(0, \infty)$.

\subsection{The multiscale chi-squared test}
Following \citet{ingster2000adaptive}, we propose a multiscale chi-squared test which consists in testing at different (dyadic) bins sizes.  We focus again on the case where $m = n$ and propose rejecting for large values of
\beq\label{adapt}
\phi = \max_\kappa \ \IND{\Gamma_\kappa \ge \tau_\kappa},
\eeq
where $\Gamma_\kappa$ is defined in \eqref{Gamma2}, 
\beq
\tau_\kappa := 2m + a m (\log m)^{1/2} \kappa^{-d/2},
\eeq
and the maximum is over $\kappa \in \{2^b : b = 1, \dots, b_{\rm max}\}$ where $b_{\rm max} := \lfloor \frac2d \log_2(m) \rfloor$.

\begin{prp} \label{prp:adapt}
The multiscale chi-squared test $\phi$ above satisfies
\beq\label{upper-adapt}
\sup_{s > 0} R_{\eps_s}^{(m,n)}(\phi; \cH_s^d(L)) \le c_1/a^2, \quad \text{if }\eps_s \ge c_2(s) a (\log m)^{1/2} m^{-2s/(4s+d)} \text{ for all } s > 0,
\eeq
where $c_1$ is a constant depending only on $(d,L)$ and $c_2(s)$ is the $c_2$ constant appearing in \eqref{chi2-1}.
\end{prp}

\citet{ingster2000adaptive} considers a different variant of the multiscale chi-squared test for which he establishes a sharper bound for the case where $\mathfrak{S}$ is a compact interval of $(0, \infty)$ using a variant of the Berry-Esseen inequality.  The resulting bound matches the minimax rate (for the adaptive setting).  We conjecture that his results generalize to higher dimensions, while our work leaves a gap between the lower bound \eqref{lower-adapt} and the upper bound \eqref{upper-adapt} of order $\sqrt{\log m}$.

\section{Proofs}
\label{sec:proofs}

\subsection{Proof of \lemref{reduction}}
Let $\phi$ be any test for the two-sample problem.  Define a test $\psi$ for the one-sample problem as follows
\beq
\psi(x_1,\dots,x_m) = \E^{(n)}_{f_0}[\phi(x_1,\dots,x_m; Y_1, \dots, Y_n)],
\eeq
where in the expectation $Y_1,\dots,Y_n$ are IID from $f_0$.  By Tonelli's theorem, we then have
\beq\begin{split}
R_\eps^{(m)}(\psi; f_0; \cH) 
&= \E^{(m,n)}_{f_0,f_0}[\phi] + \sup\Big\{\E^{(m,n)}_{f,f_0}[1-\phi] : f \in \cH; \delta(f,f_0) \ge \eps\Big\} \\
&\le R_\eps^{(m,n)}(\phi; \cH).
\end{split}\eeq
By definition $R_\eps^{(m)}(f_0; \cH) \le R_\eps^{(m)}(\psi; f_0; \cH)$, and then the proof follows by taking the infimum over all tests $\phi$ for the two-sample problem.

\subsection{Proof of \thmref{lower1}}
As is customary, we build a prior on the set of alternatives.
Let $h : \bbR^d \mapsto \bbR$ be infinitely differentiable  with support in $[0,1]^d$ and such that $\int h = 0$ and $\int h^2 = 1$.  Let $\kappa \ge 1$ be an integer.  For $j \in \bbZ^d$, define $h_{j,\kappa}(x) = \kappa^{d/2} h(\kappa x - j+1)$ and note that $h_{j,\kappa}$ is supported on $[(j-1)/\kappa, j/\kappa]$ with $\|h_{j,\kappa}\|_2 = 1$.  For $\eta = (\eta_1, \dots, \eta_{\kappa^d}) \in \{-1,+1\}^{\kappa^d}$ and $\rho > 0$ to be chosen later, define
\beq
f_\eta = f_0 + \rho \sum_{j\in [\kappa]^d} \eta_j h_{j,\kappa} \triangleq 1 + \rho \sum_{j\in [\kappa]^d} \eta_j h_{j,\kappa},
\eeq
Note that $\int f_\eta = 1$. In addition, using the fact that the $h_{j,\kappa}$'s have disjoint supports, we have the following:
\bitem
\item For $\rho \kappa^{d/2} \|h\|_\infty \le 1$, we have
\beq
f_\eta \ge 1 - \rho \kappa^{d/2} \|h\|_\infty \ge 0.
\eeq
\item For $\rho \kappa^{d/2+s} C/L \le 1$, where 
$C := 4\|h^{(\lfloor s \rfloor)}\|_\infty  \vee 2 \|h^{(\lfloor s \rfloor+1)}\|_\infty $,
we have $f_\eta \in \cH_s^d(L)$.
To see this, first note that $f_\eta$ is infinitely differentiable.
Take $x, y \in [0,1]^d$.  Let $k,l \in [\kappa]^d$ be such that $x \in [(k-1)/\kappa, k/\kappa]$ and $y \in [(l-1)/\kappa, l/\kappa]$.  Then 
\beq\begin{split}
|f_\eta^{(\lfloor s \rfloor)}(x) - f_\eta^{(\lfloor s \rfloor)}(y)| 
&= \rho \kappa^{d/2} \kappa^{\lfloor s \rfloor} \Big|\sum_{j\in [\kappa]^d} \eta_j \big[h^{(\lfloor s \rfloor)}(\kappa x -j+1) - h^{(\lfloor s \rfloor)}(\kappa y -j+1)\big]\Big| \\
&\le \rho \kappa^{d/2+\lfloor s \rfloor} \Big[\big|h^{(\lfloor s \rfloor)}(\kappa x -k+1) - h^{(\lfloor s \rfloor)}(\kappa y -k+1))\big| \\
&\quad + \big|h^{(\lfloor s \rfloor)}(\kappa x -l+1) - h^{(\lfloor s \rfloor)}(\kappa y -l+1))\big|\Big] \\
&\le \rho \kappa^{d/2+\lfloor s \rfloor} \Big[(4\|h^{(\lfloor s \rfloor)}\|_\infty) \wedge (2 \|h^{(\lfloor s \rfloor+1)}\|_\infty \kappa \|x-y\|)\Big] \\
&\le \rho \kappa^{d/2+\lfloor s \rfloor}  \big( 4\|h^{(\lfloor s \rfloor)}\|_\infty  \vee 2 \|h^{(\lfloor s \rfloor+1)}\|_\infty \big) \left( 1 \wedge \kappa \|x-y\| \right) \\
&\le L \|x - y\|^{s-\lfloor s \rfloor},
\end{split}\eeq
where we have used the fact that $(1 \wedge u)^a \leq u^a$ for any $u>0$ and $0<a<1$.
\eitem
Given $\kappa$, take $\rho$ just small enough that these conditions are satisfied.
Then with $\eps=\rho \kappa^{d/2}$, we have
\beq
\|f_\eta - f_0\|_2^2 = \rho^2 \sum_{j\in [\kappa]^d} \int (\eta_j h_{j,\kappa})^2 = \rho^2 \kappa^d = \eps^2.
\eeq  
Then, as prior on the set of alternatives, consider the uniform distribution on $\{f_\eta : \eta \in \{-1,+1\}^{\kappa^d}\}$.  In other words, the prior picks an alternative by drawing a Rademacher vector $\eta$ and forming $f_\eta$.  The following is standard.  The minimax risk is lower bounded by the Bayes (i.e., average) risk with respect to that prior, which is attained by the likelihood ratio test $\{W > 1\}$, where 
\beq
W := 2^{-\kappa^d} \sum_{\eta \in \{-1,+1\}^{\kappa^d}} \prod_{i=1}^m f_\eta(X_i).
\eeq
It is well-known that the risk of the likelihood ratio test is bounded from below by $1 - \frac12 \sqrt{\Var_0(W)}$. We thus turn to upper bounding $\Var_0(W) = \E_0(W^2) - 1$.  We have
\beq\begin{split}
\E_0(W^2)
&= 2^{-2\kappa^d} \sum_{\eta, \eta' \in \{-1,+1\}^{\kappa^d}} \prod_{i=1}^m \E_0 \big(f_\eta(X_i) f_{\eta'}(X_i)\big) 
%&= \E_\eta \E_{\eta'} \Big[\<f_\eta, f_{\eta'}\>^m\Big] \\
\end{split}\eeq
with
\beq
\prod_{i=1}^m \E_0 \big(f_\eta(X_i) f_{\eta'}(X_i)\big) = \<f_\eta, f_{\eta'}\>^m = \big(1 + \rho^2 \<\eta, \eta'\>\big)^m \le \exp\big(m \rho^2 \<\eta, \eta'\>\big),
\eeq
using the fact that $\{h_{j,\kappa} : j \in [\kappa]\}$ are orthonormal.  Thus, seeing $\eta, \eta'$ as IID Rademacher vectors, we have
\beq
\E_0(W^2) \le \E_{\eta, \eta'} \Big[\exp\big(m \rho^2 \<\eta, \eta'\>\big)\Big] = \cosh(m \rho^2)^{\kappa^d} \le \big(1 + (m \rho^2)^2\big)^{\kappa^d} \le \exp\big(\kappa^d (m \rho^2)^2\big),
\eeq
%using the fact that $\<\eta, \eta'\> \sim 2 \Bin(\kappa^d, 1/2) - \kappa^d$ and then 
using the fact that $\cosh(x) \le 1 + x^2$ for $x \in [0,1]$.  
This assumes that $m \rho^2 \le 1$.
In that case, $\Var_0(W) \le 1$ when $\kappa^d (m \rho^2)^2 \le \log 2$, which in turn implies that the minimax risk is bounded from below by $1/2$.  It is then easy to see that we may choose $\kappa = \lfloor m^{2/(4s+d)} \rfloor$ and $\rho = c m^{-(2s+d)/(4s+d)}$ for a sufficiently small constant $c>0$.  This results in $\eps \asymp m^{-2s/(4s+d)}$.

\subsection{The chi-squared test} 
\label{sec:chi2}
%We study the test based on $\Gamma_\kappa$ defined in \eqref{Gamma2} in the setting where $m = n$.  In fact, we assume without loss of generality that $m \le n$ (so that $m = m \wedge n$) and we discard at random $n-m$ points from the $Y$ sample, before apply the test based on \eqref{Gamma2}.
%We then analyze that test using Chebyshev's inequality.

We first consider the chi-squared goodness-of-fit test in a general situation, where we are testing the equality of two distributions $p$ and $q$ on a discrete set $\cK$ based on a sample $A_1, \dots, A_m$ from $p$ and a sample $B_1, \dots, B_m$ from $q$, all independent.
The (unnormalized) chi-squared rejects for large values of 
\beq
T = \sum_{k \in \cK} (M_k - N_k)^2,
\eeq
where
\beq\label{chi1}
M_k = \# \{i \in [m]: A_i = k\}, \quad
N_k = \#\{j \in [m]: B_j = k\}, \quad 
\eeq

We will see $p$ and $q$ as (probability) vectors indexed by $\cK$.
For a vector $u = (u_k)$, let $\<u\> = \sum_k u_k$, let $u^a = (u_k^a)$ for any $a > 0$, and if $v = (v_k)$ is another vector, let $uv = (u_kv_k)$.  

\begin{lem} \label{lem:chi-var}
We have
\beq
\E(T) = 2m + m^2 \<(p-q)^2\> - m (\<p^2\> + \<q^2\>),
\eeq
and
\beq
\Var(T) 
\le 2 m^2 \<(p+q)^2\> + 4 m^3 \Big(\<(p+q)(p-q)^2\> + 2\<pq\>\<(p-q)^2\>\Big).
\eeq
\end{lem}

\begin{proof}
Let $\alpha_{ii'} = \IND{A_i = A_{i'}}$, $\beta_{jj'} = \IND{B_j = B_{j'}}$, and $\gamma_{ij} = \IND{A_i = B_j}$.  We have
\begin{align}
T 
&= \sum_{i, i'} \alpha_{ii'} + \sum_{j,j'} \beta_{jj'} - 2 \sum_{i,j} \gamma_{ij} \\
&= 2m + \sum_{i \ne i'} \alpha_{ii'} + \sum_{j \ne j'} \beta_{jj'} - 2 \sum_{i,j} \gamma_{ij},
\end{align}
using the fact that $\alpha_{ii} = 1$ and $\beta_{jj} = 1$.

For the expectation, we use the fact that 
\beq
\E(\alpha_{ii'}) = \<p^2\>, \quad \forall i \ne i'; \qquad
\E(\beta_{jj'}) = \<q^2\>, \quad \forall j \ne j'; \qquad
\E(\gamma_{ij}) = \<pq\>, \quad \forall i,j,
\eeq
and use the fact that $\<(p-q)^2\> = \<p^2\> + \<q^2\> - 2\<pq\>$.

For the variance, we have
\begin{align}
\Var(T) 
&= \Var\Big(\sum_{i \ne i'} \alpha_{ii'}\Big) + \Var\Big(\sum_{j \ne j'} \beta_{jj'}\Big) + 4 \Var\Big(\sum_{i,j} \gamma_{ij}\Big) \\ 
&\quad - 4 \Cov(\sum_{i \ne i'} \alpha_{ii'},\ \sum_{i,j} \gamma_{ij}\Big) - 4 \Cov(\sum_{j \ne j'} \beta_{jj'},\ \sum_{i,j} \gamma_{ij}\Big),
\end{align}
where we used the fact that the $\alpha$'s are independent of the $\beta$'s.
We derive
\beq
\Var\Big(\sum_{i \ne i'} \alpha_{ii'}\Big) 
= \sum_{i_1 \ne i_1'}\sum_{i_2 \ne i_2'} \Cov(\alpha_{i_1i_1'}, \alpha_{i_2i_2'})
\eeq
where, for $i_1 \ne i_1'$ and $i_2 \ne i_2'$,
\begin{align}
\Cov(\alpha_{i_1i_1'}, \alpha_{i_2i_2'}) 
&= \E(\alpha_{i_1i_1'}\alpha_{i_2i_2'}) - \E(\alpha_{i_1i_1'})\E(\alpha_{i_2i_2'}) \\
&= \P(A_{i_1} = A_{i_1'}, A_{i_2} = A_{i_2'}) - \P(A_{i_1} = A_{i_1'}) \P(A_{i_2} = A_{i_2'}),
\end{align}
so that
\beq
\Cov(\alpha_{i_1i_1'}, \alpha_{i_2i_2'}) = \begin{cases}
\<p^2\> - \<p^2\>^2, & \text{if $\{i_1, i_1', i_2, i_2'\}$ has 2 distinct elements}; \\
\<p^3\> - \<p^2\>^2, & \text{if $\{i_1, i_1', i_2, i_2'\}$ has 3 distinct elements}; \\
0, & \text{otherwise}.
\end{cases}
\eeq
Counting how many instances of each case arises in the sum above, we arrive at
\begin{align}
\Var\Big(\sum_{i \ne i'} \alpha_{ii'}\Big) 
&= 2 m(m-1) (\<p^2\> - \<p^2\>^2) + 4 m(m-1)(m-2) (\<p^3\> - \<p^2\>^2) \\
&\le 2 m^2 \<p^2\> + 4 m^2(m-1) (\<p^3\> - \<p^2\>^2).
\end{align}
By analogy, we directly deduce that
\begin{align}
\Var\Big(\sum_{j \ne j'} \beta_{jj'}\Big) 
&= 2 m(m-1) (\<q^2\> - \<q^2\>^2) + 4 m(m-1)(m-2) (\<q^3\> - \<q^2\>^2) \\ 
&\le 2 m^2 \<q^2\> + 4 m^2(m-1) (\<q^3\> - \<q^2\>^2).
\end{align}
Similarly, 
\beq
\Var\Big(\sum_{i,j} \gamma_{ij}\Big)
= \sum_{i_1, j_1} \sum_{i_2, j_2} \Cov(\gamma_{i_1j_1}, \gamma_{i_2j_2}),
\eeq
with
\beq
\Cov(\gamma_{i_1j_1}, \gamma_{i_2j_2}) = \begin{cases}
\<pq\> - \<pq\>^2, & \text{if $i_1 = i_2$ and $j_1 = j_2$}; \\
\<pq^2\> - \<pq\>^2, & \text{if $i_1 = i_2$ and $j_1 \ne j_2$}; \\
\<p^2q\> - \<pq\>^2, & \text{if $i_1 \ne i_2$ and $j_1 = j_2$}; \\
0, & \text{otherwise}.
\end{cases}
\eeq
Then counting how many instances of each case arises in the sum above, we arrive at
\begin{align}
4 \Var\Big(\sum_{i,j} \gamma_{ij}\Big)
&= 4 m^2 (\<pq\> - \<pq\>^2) + 4 m^2(m-1) (\<pq^2\> + \<p^2q\> - 2\<pq\>^2) \\
&\le 4 m^2 \<pq\> + 4 m^2(m-1) (\<pq^2\> + \<p^2q\> - 2\<pq\>^2).
\end{align}
We also have
\beq
\Cov(\sum_{i \ne i'} \alpha_{ii'},\ \sum_{i,j} \gamma_{ij}\Big)
= \sum_{i_1 \ne i_1'} \sum_{i_2, j_2} \Cov(\alpha_{i_1i_1'}, \gamma_{i_2j_2}),
\eeq
with
\beq
\Cov(\alpha_{i_1i_1'}, \gamma_{i_2j_2}) = \begin{cases}
\<p^2q\> - \<p^2\>\<pq\>, & \text{if $i_1 = i_2$ or $i_1' = i_2$}; \\
0, & \text{otherwise}.
\end{cases}
\eeq
Then counting how many instances of each case arises in the sum above, we arrive at
\begin{align}
- 4 \Cov(\sum_{i \ne i'} \alpha_{ii'},\ \sum_{i,j} \gamma_{ij}\Big)
= -8  m^2(m-1)(\<p^2q\> - \<p^2\>\<pq\>).
\end{align}
By analogy, we directly deduce that
\begin{align}
-4\Cov(\sum_{j \ne j'} \beta_{jj'},\ \sum_{i,j} \gamma_{ij}\Big)
= -8 m^2(m-1)(\<pq^2\> - \<q^2\>\<pq\>).
\end{align}

Combining all the inequalities above, and simplifying a bit, we arrive at
\begin{align}
\Var(T) 
&\le 2 m^2 (\<p^2\> + \<q^2\> + 2 \<pq\>) \\
&\quad + 4 m^2(m-1) \Big(\<p^3\> - \<p^2\>^2 + \<q^3\> - \<q^2\>^2 + \<pq^2\> + \<p^2q\> - 2\<pq\>^2 \\ 
&\qquad - 2\big[\<p^2q\> - \<p^2\>\<pq\> + \<pq^2\> - \<q^2\>\<pq\>\big] \Big) \\
&= 2 m^2 (\<p^2 + q^2 + 2 pq\>) \\
&\quad + 4 m^2(m-1) \Big(\<p^3 + q^3 -p^2q -q^2p\> \\
&\qquad - \<p^2\>^2 - \<q^2\>^2 - 2\<pq\>^2 + 2\<p^2\>\<pq\> + 2\<q^2\>\<pq\>\Big) \\
&=2 m^2 \<(p+q)^2\> \\
&\quad + 4 m^2(m-1) \Big(\<(p+q)(p-q)^2\> \\
&\qquad + 2\<pq\>\<(p-q)^2\> + 2\<pq\>^2 -\<p^2\>^2 -\<q^2\>^2\Big) \\
&\le 2 m^2 \<(p+q)^2\> + 4 m^2(m-1) \Big(\<(p+q)(p-q)^2\> + 2\<pq\>\<(p-q)^2\>\Big),
\end{align}
using the fact that
\beq
2\<pq\>^2 -\<p^2\>^2 -\<q^2\>^2 \le 2 \<p^2\>\<q^2\> -\<p^2\>^2 -\<q^2\>^2 = - (\<p^2\> - \<q^2\>)^2 \le 0,
\eeq
by the Cauchy-Schwarz inequality.
\end{proof}

Having computed the mean and variance of the statistic $T$, we now apply Chebyshev's inequality to analyze the performance of the corresponding test.
For a vector $u = (u_k)$, let $\|u\|_\infty = \sup_k u_k$ and $\|u\| = (\sum_k u_k^2)^{1/2}$.

\begin{cor} \label{cor:chi-cheb}
Consider testing within the class of probability distributions $r$ on $\cK$ such that $\|r\|_\infty \le \eta$ for some $\eta \in (0, 1]$.  There are universal constants $v_1, v_2 > 0$ such that, for any $a \ge 1$, the test with rejection region $\{T - 2m \ge a m \sqrt{\eta}\}$ has size at most $v_1/a^2$, and has power at least $1-v_1/a^2$ against alternatives satisfying
\beq\label{cor-chi-cheb}
\|p-q\|^2 \ge v_2 a^2 \sqrt{\eta}/m.
\eeq
\end{cor}

\begin{proof}
First assume that we are under the null so that $p = q$.  By \lemref{chi-var}, we have $\E[T] \le 2m$ and $\Var(T) \le 8 m^2 \<p^2\>$, with
\beq
\<p^2\> = \sum_k p_k^2 \le \eta \sum_k p_k = \eta,
\eeq
using the fact that $\|p\|_\infty \le \eta$, $\min_k p_k \ge 0$, and $\sum_k p_k = 1$.  
(The same manipulations are performed below, but details are omitted.)
Thus, by Chebyshev's inequality, 
\beq
T - 2m \ge a m \sqrt{\eta}
\eeq
with probability at most $8/a^2$.
%This proves the bound on the size.

When $p \ne q$, by \lemref{chi-var}, we have 
\beq
\E[T] -2m \ge m^2 \|p-q\|^2 - 2 m\eta
\eeq
and 
\beq
\Var(T) 
\le 8 m^2 \eta + 4 m^3 [ 2 \eta \|p-q\|^2 + 2 \eta \|p-q\|^2]
= 8 m^2 \eta + 16 m^3 \eta \|p-q\|^2.
\eeq
Then, by Chebyshev's inequality,
\beq
T - 2m \le m^2 \|p-q\|^2 - 2 m\eta - a \sqrt{8 m^2 \eta + 16 m^3 \eta \|p-q\|^2}
\eeq
with probability at most $1/a^2$.  Thus the test has power at least $1-1/a^2$ against alternatives where
\beq
m^2 \|p-q\|^2 - 2 m\eta - a \sqrt{8 m^2 \eta + 16 m^3 \eta \|p-q\|^2}
\ge a m \sqrt{\eta}.
\eeq
This shows that it is enough that
\beq
m^2 \|p-q\|^2 \ge 4 \Big(2 m\eta \vee a \sqrt{8 m^2 \eta} \vee a \sqrt{16 m^3 \eta \|p-q\|^2} \vee a m \sqrt{\eta}\Big).
\eeq
From this we conclude, using the fact that $\eta \le 1$ and $a \ge 1$.
\end{proof}

\subsection{Proof of \thmref{chi2}}
\label{sec:Gamma}
Since we are dealing with continuous distributions we have to take the approximation error due to binning into account.  
Leaving $\kappa$ implicit, define
\beq\label{approx-def}
H_k = [(k-1)/\kappa, k/\kappa], \quad 
p_k = \int_{H_k} f(x) {\rm d}x, \quad 
q_k = \int_{H_k} g(x) {\rm d}x.
\eeq
By definition, for any $h \in \cH_s^d(L)$, $\|h\|_\infty \le L$.
%and 
%\beq\label{modulus}
%|h(x) - h(y)| \le L |x-y|^{s \wedge 1}, \quad \forall x,y \in [0,1]^d.
%\eeq
Using this, we get
%\beq\begin{split}
%p_k 
%&\le f(k/\kappa) \kappa^{-d} + L \sqrt{d} \kappa^{-d - s \wedge 1} \\
%&\le L \kappa^{-d} + L \sqrt{d} \kappa^{-d - s \wedge 1}
%\le \eta := 2 L \kappa^{-d}, \quad\text{for $\kappa$ large enough},
%\end{split}\eeq
\begin{equation}\label{eta}
p_k \leq \|f\|_\infty \kappa^{-d} \leq \eta := L\kappa^{-d}, \quad \forall k,
\end{equation}
and similarly for $q_k$.  
Define $p = (p_k)$ and $q = (q_k)$.  
By \corref{chi-cheb} (with $v_1$ and $v_2$ defined there), the test under consideration has risk at most $2v_1 L/a^2$ when \eqref{cor-chi-cheb} holds under the alternative.  We now examine when  \eqref{cor-chi-cheb} holds.  

The following is an extension of \cite[Eq (15)]{ingster1987minimax} to the setting of dimension $d$.  
The proof is provided in \secref{proof-approx}.

\begin{lem} \label{lem:approx}
For a continuous function $h : [0,1]^d \mapsto \bbR$ and an integer $\kappa \ge 2$, define $W_\kappa[h] = \sum_k \kappa^d (\int_{H_{k, \kappa}} h)\bbI_{H_{k, \kappa}}$, where $H_{k, \kappa} = [(k-1)/\kappa, k/\kappa]$. 
There are constants $b_1, b_2  > 0$ depending only on $(d,s,L)$ such that
\beq
\|W_\kappa[h]\|_2 \ge b_1 \|h\|_2 - b_2 \kappa^{-s}, \quad \forall 
h \in \cH_s^d(L). 
\eeq
\end{lem}

Applying \lemref{approx}, we obtain
\beq
\|p - q\|^2 = \kappa^{-d} \|W_\kappa[f] - W_\kappa[g]\|_2^2 \ge \kappa^{-d} \big(b^*_1 \|f -g\|_2 - b^*_2 \kappa^{-s}\big)^2,
\eeq
where $b^*_1 := b_1(d,s,2L)$ and $b^*_2 := b_2(d,s,2L)$, using the fact that $f-g \in \cH_s^d(2L)$ when $f, g \in \cH_s^d(L)$.

In our setting, recall that $\|f-g\|_2 \ge \eps$.  Because of our choice of $\kappa$ in \eqref{kappa}, taking the constant $c_2$ in \eqref{chi2-1} large enough, we may enforce $\eps \ge 2 (b^*_2/b^*_1) \kappa^{-s}$, and thus obtain the lower bound
\beq\label{p-q}
\|p - q\|^2 \ge c_3 \kappa^{-d} \eps^2,
\eeq
for a constant $c_3 > 0$ depending on $(s,d,L)$.
Condition \eqref{cor-chi-cheb} holds when 
\beq
c_3 \kappa^{-d} \eps^2 \ge v_2 a^2 \sqrt{\eta}/m,
\eeq
which, recalling that $\eta = L\kappa^{-d}$, is the case when 
\beq
\eps^2 \ge c_4 a^2 \kappa^{d/2}/m,
\eeq
for a constant $c_4 > 0$ depending on $(s,d,L)$.
Plugging-in the value for $\kappa$ from \eqref{kappa}, we conclude.

%\begin{rem}
%We proved that the risk is bounded by $2/a^2$, but by increasing $c_1$ and $c_2$ a little bit --- which are unspecified in the statement --- we can bound the risk by $1/a^2$ as stated.
%\end{rem}

\subsection{Proof of \lemref{approx}}
\label{sec:proof-approx}
The following is well-known.
\begin{lem} \label{lem:taylor}
Fix any $h \in \cH_s^d(L)$ and any $x_0 \in [0,1]^d$.  Let $u$ denote the $\lfloor s \rfloor$-order Taylor series of $h$ at $x_0$.  There is a constant $L'$ depending only on $(d,s,L)$ such that
\beq
|h(x) - u(x)| \le L' \|x - x_0\|^s, \quad \forall x \in [0,1]^d.
\eeq
\end{lem}

Fix $h \in \cH_s^d(L)$ and an integer $r \ge 1$ to be chosen large enough, but fixed, later on.  Let $u_j$ be the $\lfloor s \rfloor$-order Taylor series of $h$ at $(j-1)r/\kappa$.  Assume for convenience that $\kappa/r$ is an integer.
For a positive integer $j_0 \le \kappa/r$, let $\tilde H_{j_0}$ denote $[(j_0-1)r/\kappa, j_0 r/\kappa]$, and for $j = (j_1, \dots, j_d)$, with some abuse of notation, define $\tilde H_j = \prod_{l=1}^d \tilde H_{j_l}$.  Define $u = \sum_j u_j \bbI_{\tilde H_j}$.  By \lemref{taylor}, we have
\beq
|h(x) - u(x)| \le L' (\sqrt{d} 2r/\kappa)^s =: a_1 \kappa^{-s}, \quad \forall x \in [0,1]^d,
\eeq
where $a_1$ does not depend on $\kappa$.  (We used the fact that $\diam(\tilde H_j) \le \sqrt{d} 2r/\kappa$.)
With this and the triangle inequality, we have
\beq\label{W-lb}\begin{split}
\|W_\kappa[h]\|_2 
&\ge \|W_\kappa[u]\|_2 -\|W_\kappa[u] - W_\kappa[h]\|_2 \\
&\ge \|W_\kappa[u]\|_2 -\|u - h\|_2 \\
&\ge \|W_\kappa[u]\|_2 -a_1 \kappa^{-s}.
\end{split}\eeq

We need the following.

\begin{lem}\label{lem:proj}
Let $\cP^d_p$ denote the class of polynomials on $\bbR^d$ of degree at most $p$.  Let $\cQ = (Q_j)$ be a discrete partition of $[0,1]^d$ such that $|Q_j| > 0$ for all $j$ and define $W_\cQ[v] = \sum_j \frac1{|Q_j|}(\int_{Q_j} v) \bbI_{Q_j}$ for any continuous function $v: [0,1]^d \to \bbR$.  There is a constant $a_2 > 0$ depending only on $(d, p)$ such that, when $\diam(\cQ) := \max_j \diam(Q_j) \le a_2$, 
\beq
\|W_\cQ(v)\|_2 \ge a_2 \|v\|_2, \quad \forall v \in \cP^d_p.
\eeq
\end{lem}

\begin{proof}
Assume for contradiction that there is an sequence $\cQ^l = (Q_j^l)$ of partitions of $[0,1]^d$ such that, for all $l \ge 1$, $\diam(\cQ^l) \le 1/l$ and there is some $v_l \in \cP^d_p$ such that $\|W_{\cQ^l}(v_l)\|_2 < (1/l) \|v_l\|_2$.  Without loss of generality because of homogeneity, we may assume that $\|v_l\|_2 = 1$.  In that case, by compactness of $\{v \in \cP^d_p : \|v\|_2 = 1\}$, the sequence $(v_l)$ has at least one accumulation point.\footnote{ The topology induced by the $L^2([0,1]^d)$ on $\cP_p^d$ coincides with its topology as a finite vector space.}
Choose one of them, denoted $v_\infty \in \cP^d_p$.  
On the one hand, $\|v_\infty\|_2 =1$.  
On the other hand, 
\beq\begin{split}
\|W_{\cQ^l}(v_\infty)\|_2 
&\le \|W_{\cQ^l}(v_\infty - v_l)\|_2 + \|W_{\cQ^l}(v_l)\|_2 \\
&\le \|v_\infty - v_l\|_2 + 1/l \to 0, \quad l \to \infty,
\end{split}\eeq
and
\beq
\|W_{\cQ^l}(v_\infty)\|_2 \to \|v_\infty\|_2, \quad l \to \infty, 
\eeq
by dominated converge, and together this implies that $\|v_\infty\|_2 = 0$.
We thus have a contradiction.
\end{proof}

Coming back to \eqref{W-lb}, since $W_\kappa[u] = \sum_j W_\kappa[u_j \bbI_{\tilde H_j}]$ with the functions on the RHS having non-overlapping supports, we have
\beq
\|W_\kappa[u]\|_2^2 = \sum_j \|W_\kappa[u_j \bbI_{\tilde H_j}]\|_2^2.
\eeq
By \lemref{proj}, after translating and rescaling $\tilde H_j$ into $[0,1]^d$, we have $\|W_\kappa[u_j \bbI_{\tilde H_j}]\|_2 \ge a_2 \|u_j \bbI_{\tilde H_j}\|_2$ as soon as $r$ is large enough (independent of $\kappa$).  Here $a_2$ is the constant of \lemref{proj} corresponding to $(d, \lfloor s \rfloor)$ and we have used the fact that $u_j$ is a polynomial on $\tilde H_j$ of degree at most $\lfloor s \rfloor$.  
We fix such an $r$ henceforth, and derive
\beq
\|W_\kappa[u]\|_2^2 \ge \sum_j a_2^2 \|u_j \bbI_{\tilde H_j}\|_2^2 = a_2^2 \|u\|_2^2.
\eeq
And, as before, 
\beq
\|u\|_2 \ge \|h\|_2 - \|h -u\|_2 \ge \|h\|_2 - a_1 \kappa^{-s}.
\eeq
Thus, 
\beq\begin{split}
\|W_\kappa[h]\|_2 
&\ge a_2 (\|h\|_2 - a_1 \kappa^{-s}) -a_1 \kappa^{-s} \\
& \ge a_2 \|h\|_2 - (2 a_2 +1) a_1 \kappa^{-s},
\end{split}\eeq
concluding the proof of \lemref{approx}.

\subsection{Proof of \prpref{adapt}}
We elaborate on the arguments laid out in \secref{Gamma}.  
By \corref{chi-cheb} (with $v_1$ and $v_2$ defined there), the test with bin size $1/\kappa$ has (statistical) size at most $v_1 L/(a \sqrt{\log m})^2$, so by the union bound, the combined test \eqref{adapt} has size at most $b_{\rm max} v_1 L/(a^2 \log m) \le c_1/a^2$, for a constant $c_1$ that depends only on $(d,L)$.

Now fix an alternative, meaning with $f \ne g$ both in $\cH_d^s(L)$ with such that $\|f - g\|_2 \ge \eps$.  Let $b$ be the integer defined by $2^{b-1} < \kappa_s \le 2^b$, where $\kappa_s = \kappa(s,d)$ is defined in \eqref{kappa}.  Note that (eventually) $1 \le b \le b_{\rm max}$ by the fact that $b -1 < \log_2 (\kappa_s) \le b$ with $\log_2 (\kappa_s) = \frac2{4s+d} \log_2 m$ and $\frac2{4s+d} < \frac2d$.
We then let $\kappa = 2^b$ and consider the test based on $\Gamma_\kappa$.  The same analysis carried out in \secref{Gamma} applies as such if in place of \eqref{eta} we use the fact that $p_k \le L\kappa_s^{-d} \le 2^{-d} L \kappa^{-d}$, because $\kappa/2 \le \kappa_s \le \kappa$.

\subsection*{Acknowledgments}
We are grateful to Yannick Baraud for helpful discussions and to Aaditya Ramdas for alerting us about his work on the topic \cite{ramdas2015adaptivity, ramdas2015decreasing, reddi2015high}.  
We also acknowledge the contributions of two anonymous reviewers.  
This work was partially supported by the US National Science Foundation (DMS 1223137, DMS 1513465, and CCF 1320547).

\bibliographystyle{abbrvnat}
\bibliography{gof-graph}

\end{document}